\documentclass[12pt,leqno]{amsart}
\usepackage{amssymb}
\usepackage{amsmath}
\usepackage{amsthm}
\usepackage{amsfonts}
\usepackage{mathrsfs} 
\usepackage{float}
\usepackage{graphicx}

\newcounter{minutes}
\divide\time by 60
\newcounter{hours}
\multiply\time by 60
\addtocounter{minutes}{-\time}
\usepackage{enumerate}
\usepackage[T1]{fontenc} 

\theoremstyle{plain}
\newtheorem{theorem}{Theorem}[section]
\newtheorem{lemma}[theorem]{Lemma}
\newtheorem{corollary}[theorem]{Corollary}

\theoremstyle{definition}

\DeclareMathOperator{\esssup}{ess\,sup}
\DeclareMathOperator{\essinf}{ess\,inf}
\DeclareMathOperator{\Real}{Re \,}
\DeclareMathOperator{\Imaginary}{Im \,}
\DeclareMathOperator{\Arg}{Arg \,}

\numberwithin{equation}{section}

\begin{document}

\title[Circular Symmetrization and Arclength Problems]{%
Circular Symmetrization, Subordination and Arclength problems
on Convex Functions}
\date{}

\author{M. Okada}
\address{M. Okada,
Department of Applied Science,
Faculty of Engineering,
Yamaguchi University,
Tokiwadai, Ube 755-8611, Japan.}
\email{okada@yamaguchi-u.ac.jp}

\author{S. Ponnusamy}
\address{S. Ponnusamy,
Indian Statistical Institute (ISI), Chennai Centre,
SETS (Society for Electronic Transactions and security),
MGR Knowledge City, CIT Campus, Taramani,
Chennai 600 113, India.}
\email{samy@isichennai.res.in, samy@iitm.ac.in}

\author{A. Vasudevarao}
\address{A. Vasudevarao,
Department of Mathematics,
Indian Institute of Technology Kharagpur,
Kharagpur-721 302, West Bengal, India.}
\email{alluvasu@maths.iitkgp.ernet.in}

\author{H. Yanagihara}
\address{H. Yanagihara,
Department of Applied Science,
Faculty of Engineering,
Yamaguchi University,
Tokiwadai, Ube 755-8611, Japan.}
\email{hiroshi@yamaguchi-u.ac.jp}


\thanks{
${}^\dagger$ {\tt This authors is on leave from   
Indian Institute of Technology Madras,  India.}
}

\def\thefootnote{}
\footnotetext{ {\tiny File:~\jobname.tex,
printed: \number\year-\number\month-\number\day,
          \thehours.\ifnum\theminutes<10{0}\fi\theminutes }
} \makeatletter\def\thefootnote{\@arabic\c@footnote}\makeatother

\begin{abstract}
We study the class ${\mathcal C}(\Omega )$
of univalent analytic functions
$f$
in the unit disk
$\mathbb{D} = \{ z \in \mathbb{C} :\,|z|<1 \}$ of the form
$f(z)=z+\sum_{n=2}^{\infty}a_n z^n$
satisfying
\[
  1+\frac{zf''(z)}{f'(z)} \in \Omega ,
 \quad z\in \mathbb{D},
\]
where $\Omega$ will be a proper
subdomain of ${\mathbb C}$ which is starlike
with respect to $1 ( \in \Omega)$.
Let $\phi_\Omega$ be the unique conformal mapping
of ${\mathbb D}$ onto $\Omega$ with $\phi_\Omega (0)=1$ and $\phi_\Omega '(0) > 0$
and $ k_\Omega (z) = \int_0^z \exp \left( \int_0^t
\zeta^{-1} (\phi_\Omega ( \zeta) -1 ) \, d \zeta \right) \, dt$.
Let $L_r(f)$ denote the arclength of the image
of the circle $\{z \in \mathbb{C} : \, |z|=r\}$, $r\in (0,1)$.
The first result in this paper
is an inequality $L_r(f) \leq L_r(k_\Omega )$ for
$f \in \mathcal{C} ( \Omega )$, which
solves the general extremal problem
$\max_{f \in {\mathcal C}(\Omega)} L_r(f)$,
and contains many other well-known results
of the previous authors as
special cases.
Other results of this article cover
another set of related problems
about integral means in the general setting
of the class ${\mathcal C}(\Omega )$.
\end{abstract}
\thanks{}
\dedicatory{}

\maketitle

\section{Introduction}
Let ${\mathbb C}$ be the complex plane
and ${\mathbb D}(c,r) = \{ z \in {\mathbb C} :\, |z-c| < r \}$
with $c \in {\mathbb C}$ and $r > 0$.
In particular we  denote the unit disk ${\mathbb D}(0,1)$ by ${\mathbb D}$.
Let ${\mathcal A} $
be the linear space of all analytic functions
in the unit disk ${\mathbb D}$,
endowed with the topology of
uniform convergence on every compact subset of
${\mathbb D}$.
Set $ \mathcal{A}_0 = \{ f \in {\mathcal H} :\, f(0)=f'(0)-1 = 0\}$
and denote by
${\mathcal S}$ the subclass of $\mathcal{A}_0$
consisting of all univalent functions as usual.
Then ${\mathcal S}$
is a compact subset of the metrizable space $\mathcal{A}$.
See \cite[Chap. 9]{DurenBook} for details.
For $f \in \mathcal{A}$ and $0<r<1$, let
\[
    L_r(f)
    =
    \int_{-\pi}^\pi
    r |f'(re^{i \theta})| \, d \theta
\]
denote the arclength
of the image of the circle
$\partial {\mathbb D}(0,r)= \{z \in \mathbb{C} : \, |z|=r\}$.
Many extremal problems in the class
${\mathcal S}$ have been solved by the Koebe function
\[
   k(z) = \frac{z}{(1-z)^2}
\]
or by its rotation: $k_\theta  (z) = e^{-i\theta } k(e^{i\theta } z)$,
where $\theta$ is real.
Note that
 $k_\theta $ maps the unit disk $\mathbb{D}$
onto the complement of a ray.
In any case, since the functional ${\mathcal A} \ni f \mapsto L_r(f)$
is continuous and the class
${\mathcal S}$ is compact,
a solution of the extremal problem
\begin{equation}
\label{eq-ext1}
      \max_{f \in {\mathcal S} } L_r(f)
\end{equation}
exists and is in ${\mathcal S}$.
We remark that with a clever use of Dirichlet-finite integral
and the isoperimetric inequality,
Yamashita \cite{Yama-1990} obtained
the  following upper and lower estimates for the functional \eqref{eq-ext1}:
\[
   m(r)
   \leq
   L_r(k)
   \leq
   \max_{f \in {\mathcal S} } L_r(f)
   \leq
   \frac{2\pi r}{(1-r)^2} ,
\]
where
\[
 m(r)=  \frac{2\pi r\sqrt{r^4+4r^2+1}}{(1-r^2)^2}
 \geq
 \frac{2\pi r(1+r)}{(1-r)^2} \frac{\sqrt{6}}{8}
 >
 \frac{\pi r(1+r)}{2(1-r)^2}.
\]
This observation provides an improvement
over
the earlier result of Duren \cite[Theorem 2]{Duren}
and \cite[p.~39]{Du}, and moreover,
\[
  m(r)\geq  \frac{\sqrt{6}}{2}\frac{\pi r}{(1-r)^2}.
\]
The extremal problem
\eqref{eq-ext1}
stimulated much research in the theory of univalent functions,
and the problem of determining of the maximum
value and the extremal functions in ${\mathcal S}$ remains open.
However, the extremal problem
\begin{equation}
\label{eq-ext2}
   \max_{f \in {\mathcal F}} L_r(f)
\end{equation}
has been solved for
a number of subclasses $\mathcal F$ of ${\mathcal S}$.
In order to motivate these known results
and also for our further discussion on this topic,
we need to introduce some notations.

Unless otherwise stated explicitly,
throughout the discussion $\Omega$ will be a simply connected domain
in ${\mathbb C}$ with $1 \in \Omega \not= {\mathbb C}$
and  $\phi _\Omega$ is the unique conformal mapping
of ${\mathbb D}$ onto $\Omega$
with $\phi_\Omega (0)=1$ and $\phi_\Omega '(0) > 0$.

Ma and Minda \cite{Ma-Minda} considered
the classes ${\mathcal S}^*(\Omega )$ and $\mathcal{C}(\Omega )$
with some mild conditions,
eg. $\Omega $ is starlike with respect to $1$ and
the symmetry with respect to the real axis $\mathbb{R}$, i.e.,
$\overline{\Omega} = \Omega$:
\[
  {\mathcal S}^* (\Omega)
  =
  \left\{
  f \in \mathcal{A}_0 :\,
  \frac{zf'(z)}{f(z)} \in \Omega
  \text{ on $\mathbb{D}$}
  \right\} ,
  \]
and
\[
  \mathcal{C}(\Omega)
  =
  \left\{ f \in \mathcal{A}_0 : \,
  1+\frac{zf''(z)}{f'(z)} \in \Omega  \text{ on $\mathbb{D}$}  \right\} .
\]
Note that, with the special choice of
$\Omega = \mathbb{H} = : \{ w \in \mathbb{C} : \Real w > 0 \}$,
these two classes consist of starlike and convex functions in the standard sense,
and are denoted simply by ${\mathcal S}^*$ and $\mathcal{C}$,  respectively.

If $0< \alpha \leq 1$ and
$\Omega  = \{ w \in \mathbb{C} : | \Arg w | < 2^{-1} \pi \alpha \}$,
then  $\phi_\Omega (z) = \{ (1+z)/(1-z) \}^\alpha $,
and hence, in this choice ${\mathcal C} (\Omega ) $
reduces to the class of strongly convex functions
of order $\alpha$. 

Furthermore,  for $-1/2 \leq \beta < 1$
and $\Omega = \{w\in \mathbb{C}:\, \text{Re} \,  w > \beta  \}$
and $\phi _\Omega (z) = (1+(1-2\beta)z)/(1-z)$,
the class ${\mathcal C} ( \{w\in \mathbb{C}:\, \text{Re} \,  w > \beta  \} ) $
coincides with the class of convex functions of order $\beta$.
Various subclasses of ${\mathcal C }$
can be expressed in this way.
For details we refer to  \cite{Ma-Minda}
and \cite{Yanagihara}.
We  notice that it may be possible that
$\mathbb{H} \subset \Omega $,
and in this case we have
${\mathcal C} \subset {\mathcal C} ( \Omega ) $ whenever  $0\leq \beta < 1$. When $-1/2 \leq \beta < 0$,
functions in ${\mathcal C} ( \Omega )$ are known to be convex in some direction (see
\cite{Umezawa}).

A function $f$ in $\mathcal{A}_0$ is said
to be close-to-convex if there exists
a convex function $g$
and a real number $\beta\in (-\pi/2,\pi/2)$ such that
\[
  e^{i \beta }
  \frac{ f'(z)}{g'(z)}
  \in \mathbb{H}
\]
on ${\mathbb D}$.
We denote the
class of close-to-convex functions in $\mathbb{D}$ by $\mathcal{K}$
which has been introduced by Kaplan  \cite{Kaplan52}.
These standard geometric classes are related by the proper inclusions
\[
   \mathcal{C}
   \subsetneq
   {\mathcal S}^*
   \subsetneq
   \mathcal{K}
   \subsetneq
   {\mathcal S} .
\]
The extremal problem  \eqref{eq-ext2}
for ${\mathcal F}= {\mathcal C}$ has been solved
by Keogh \cite{Keogh} who showed that
\[
  \max_{f \in {\mathcal C}} L_r(f)
  =
  \frac{2\pi r}{1-r^2} = L_r( \ell_\theta )
\]
with equality if and only if
$f= \ell_\theta$.
Here   $\ell_\theta (z)=z/(1- e^{i\theta }z)$, where $\theta$ is real.
The extremal problem  \eqref{eq-ext2}
for the cases ${\mathcal F}= {\mathcal S}^*$
and  ${\mathcal F}= {\mathcal K}$
were solved by Marx \cite{marx32}
and Clunie and Duren \cite{Clunie-Duren}, respectively.
In both cases, the Koebe function and its rotations
solve the corresponding extremal problem.
That is,
for ${\mathcal F}= {\mathcal S}^*$
and ${\mathcal F}= {\mathcal K}$,
one has
$\max_{f \in {\mathcal F}} L_r(f) = L_r(k)$
with equality if and only if $f(z)=k_\theta  (z) $,
where $\theta$ is real.
As a  straightforward adaptations of the known proofs,
Miller  \cite{miller72}
extended all these three cases to
the corresponding subclasses of ${\mathcal S}$ consisting of
$m$-fold convex, starlike and close-to-convex functions, respectively.
Finally, by making use of the theory of symmetrization developed by
Baernstein \cite{Baernstein},
Leung \cite{Leung} extended the result
for ${\mathcal F}= {\mathcal S}^*$
to the class of Bazilevi\v{c} functions
and the generalized functional
$\int_{-\pi}^\pi \Phi ( \log |f'(re^{i \theta })|)\, d \theta $,
where $\Phi$ is a nondecreasing convex function on ${\mathbb R}$.
Recently, the extremal problem  \eqref{eq-ext2} for the class  of convex functions
of order $-1/2$ was solved in \cite{AbuLiPo}.

One of the aims of the present article is
to study similar extremal problems
for various subclasses
$\mathcal{C} ( \Omega )$
in a unified manner.
Let
\begin{equation}
\label{eq-ext3}
   k_\Omega(z)
   =
   \int_0^z
   \exp \left(
   \int_0^{t} \frac{\phi_\Omega ( \zeta )-1}{\zeta} \, d \zeta
    \right) d t ,
   \quad z \in {\mathbb D} .
\end{equation}
Then $k_\Omega \in {\mathcal C} (  \Omega )$.
When $\Omega$ is starlike with respect to $1$,
the extremal problem
$\max_{f \in {\mathcal C}( \Omega)} L_r(f) $
can be solved
and $k_\Omega$ plays
the role of the extremal function.

\begin{theorem}
\label{thm-01}
If $\Omega$ is starlike with respect to $1$, then, for
$f \in {\mathcal C}( \Omega )$, we have
\begin{equation}
L_r(f) \leq L_r( k_\Omega)
\label{ineq:arclenght}
\end{equation}
with equality if and only if
$f(z) =  \overline{\varepsilon} k_\Omega( \varepsilon z)$
for some $\varepsilon \in \partial \mathbb{D}$.
\label{thm:arclength}
\end{theorem}

Let $f$ and $F$ be analytic functions in ${\mathbb D}$.
Then $f$ is said to be subordinate to $F$ ($f \prec F$,
or $f(z)\prec F(z)$ in $\mathbb{D}$, in short)
if there exists an analytic function $\omega$ in ${\mathbb D}$
with $|\omega (z)| \leq |z|$ and $f(z)=F(\omega (z))$ in ${\mathbb D}$.
In particular $f({\mathbb D}) \subset F({\mathbb D})$ holds, if $f \prec F$.
Notice that when $F$ is univalent in ${\mathbb D}$,
$f \prec F$ if and only if $f({\mathbb D}) \subset F({\mathbb D})$
and $f(0)=F(0)$.

Furthermore, by making use of subordination
and circular symmetrization,
we can considerably strengthen Theorem \ref{thm-01}.
We note  that $f \in {\mathcal C} ( \Omega ) $
forces that $f'(z) \not= 0$ in ${\mathbb D}$ and
the single valued branch $\log f'(z) $ with
$\log f'(0) = 0$ exists on ${\mathbb D}$.

\begin{theorem}
If $\Omega$ is starlike with respect to $1$,
then $\log k_\Omega'(z)$ is convex univalent in ${\mathbb D}$
and $\log f' (z) \prec \log k_\Omega'(z) $ holds for
$f \in {\mathcal C}( \Omega )$.
Furthermore for any subharmonic function
$u$ in the domain $ \log k_\Omega ' ( {\mathbb D} )$ and $r \in (0,1)$
\[
  \int_{-\pi}^\pi u( \log f' (re^{i \theta })) \, d \theta
  \leq
 \int_{-\pi}^\pi u( \log k_\Omega' (re^{i \theta })) \, d \theta
\]
holds with equality for some $u$ and $r \in (0,1)$
if and only if
$u$ is harmonic in $\log k_\Omega' ({\mathbb D}(0,r))$ or
$f (z) = \overline{\varepsilon} k_\Omega( \varepsilon z )$
for some $\varepsilon \in \partial {\mathbb D}$.
\label{MainTheorem}
\end{theorem}

By letting $u(w)$
as particular functions we can obtain various inequalities.
We shall only give typical examples. Since the functions
$\log |w| $, $|w|^p$ with $0 < p < \infty$,
$\Phi ( \pm \text{\rm Re} \, w )$ or
$\Phi ( \pm \text{\rm Im} \, w )$ with a continuous convex function
$\Phi$ on ${\mathbb R}$ are subharmonic functions
of $w \in {\mathbb C}$,
we have the following inequalities.

\begin{corollary}
If $\Omega$ is starlike with respect to $1$,
then for any $f \in {\mathcal C} ( \Omega )$ and $r \in (0,1)$
the following inequalities hold.
\begin{align}
& \int_{-\pi}^\pi
\log | \log f'(re^{i \theta }) |  \, d \theta
\leq   \int_{-\pi}^\pi
\log | \log k_\Omega'(re^{i \theta }) |  \, d \theta ,
\label{ineq:loglog}
\\
 &  \int_{-\pi}^\pi | \log f'(re^{i \theta }) |^p  \, d \theta
 \leq
 \int_{-\pi}^\pi | \log k_\Omega'(re^{i \theta }) |^p  \, d \theta 
\quad (0 < p < \infty),
\label{ineq:p-th-log}
\\
 & \int_{-\pi}^\pi \Phi ( \pm \log |f'(re^{i \theta })| ) \, d \theta
\leq
 \int_{-\pi}^\pi
 \Phi ( \pm \log |k_\Omega'(re^{i \theta })| ) \, d \theta ,
\label{ineq:Philog}
\\
& \int_{-\pi}^\pi \Phi ( \pm \arg f'(re^{i \theta }) )  \, d \theta
 \leq  \int_{-\pi}^\pi  \Phi ( \pm \arg k_\Omega'(re^{i \theta }) )
 \, d \theta .
\label{ineq:Phiarg}
\end{align}
Equality holds in $(\ref{ineq:loglog})$
or $(\ref{ineq:p-th-log})$ if and only if
$f(z) =  \overline{\varepsilon} Q_\Omega( \varepsilon z)$
for some $\varepsilon \in \partial {\mathbb D}$.
Furthermore when $\Phi (\pm t )$ is not  linear
in the interval
\begin{equation}
   \left(
     \min_{- \pi \leq \theta \leq \pi }
     \log |k_\Omega' (re^{i \theta })| ,
     \max_{- \pi \leq \theta \leq \pi }
     \log |k_\Omega' (re^{i \theta })|
   \right)
\label{log-interval}
\end{equation}
or
\begin{equation}
  \left(
    \min_{- \pi \leq \theta \leq \pi }
    \arg k_\Omega' (re^{i \theta }) ,
    \max_{- \pi \leq \theta \leq \pi }
    \arg k_\Omega' (re^{i \theta })
  \right) ,
\label{arg-interval}
\end{equation}
equality holds respectively
in $(\ref{ineq:Philog})$ or $(\ref{ineq:Phiarg})$
if and only if
$f(z) =  \overline{\varepsilon} k_\Omega( \varepsilon z)$
for some $\varepsilon \in \partial \mathbb{D}$.
\label{cor:several-ineqalities}
\end{corollary}

In contrast to the above corollary
we need to assume that $\Phi$ is nondecreasing
in the following theorem.

\begin{theorem}
If $\Omega$ is starlike with respect to $1$ and
symmetric with respect to ${\mathbb R}$,
then $\log | k_\Omega' (re^{i \theta })|$ is
a symmetric function of $\theta $ and nonincreasing on $[0, \pi ]$,
and for any convex and nondecreasing function
$\Phi$ in ${\mathbb R}$
and any Lebesgue measurable set $E \subset [-\pi,\pi]$
of Lebesgue measure $2 \theta$, we have
\[  \int_{E} \Phi ( \log |f'(re^{i s })| )  \, d s
\leq  \int_{- \theta }^\theta
\Phi ( \log |k_\Omega'(re^{i s })| ) \, d s .
\]
In particular
\[\int_{E} r |f'(re^{i s })|   \, d \theta
\leq   \int_{- \theta }^\theta r |k_\Omega'(re^{i s })|  \, d s ,
\]
i.e., the length of $\{ f(re^{i s}) :\, s \in E \}$
does not exceed that of
$\{ k_\Omega(re^{i s }) :\, |s | \leq \theta \}$.
\label{thm:localization}
\end{theorem}

\section{Subordination}
First we state a variant of the Littlewood subordination
theorem  (see \cite[Theorem 1.7]{DurenBook})
and give a proof for completeness.

\begin{lemma}
\label{lemma:subordination_and_mean}
Let $f$, $F \in \mathcal{A} $ with $f \prec F$.
Then for any subharmonic function
$u$ in $F({\mathbb D})$ and $r \in (0,1)$
\begin{equation}
\int_{-\pi}^\pi u(f(re^{i \theta }) \, d \theta
\leq  \int_{-\pi}^\pi u (F (re^{i \theta })) \, d \theta
\label{ineq:Littlewood}
\end{equation}
with equality if and only if $f(z) = F( \varepsilon z )$
for some $\varepsilon \in \partial {\mathbb D}$
or $u$ is harmonic in $F( {\mathbb D} (0,r ))$.
\end{lemma}
\proof
Let $\omega \in \mathcal{A}$
with $|\omega (z)| \leq |z|$ and $f(z)=F(\omega (z))$ in ${\mathbb D}$.
Let $U$ be the continuous function on $\overline{\mathbb D}(0,r)$
such that $U =u \circ F $ on $\partial {\mathbb D}(0,r)$
and harmonic in ${\mathbb D}(0,r)$.
Since $u \circ F $ is subharmonic in ${\mathbb D}$,
it follows from the maximum principle that
$(u \circ F) (z) \leq U(z)$ on $\overline{\mathbb D}(0,r)$.
Thus
\[
   (u \circ f) (z)
   = (u \circ F \circ \omega) (z)
   \leq U \circ \omega (z)
\]
and
\begin{align*}
  \int_{-\pi}^\pi u (f (re^{i \theta } )) \, d \theta
  &=
 \int_{-\pi}^\pi (u \circ F \circ \omega)( re^{i \theta } )
\, d \theta
\\
   &\leq
 \int_{-\pi}^\pi (U \circ \omega ) ( re^{i \theta } )  \, d \theta
\\
  &=
 2 \pi (U \circ \omega)(0)
 = 2 \pi U(0)
 = \int_{-\pi}^\pi U(  re^{i \theta } )  \, d \theta
 = \int_{-\pi}^\pi u (F (re^{i \theta } ) ) \, d \theta .
\end{align*}

If $f(z) = F( \varepsilon z )$
for some $\varepsilon \in \partial {\mathbb D}$,
then equality trivially holds in
(\ref{ineq:Littlewood}).
Also if $u$ is harmonic in $F( {\mathbb D} (0,r ))$,
then $u \circ F$ and $u \circ f$ are
harmonic in ${\mathbb D} (0,r )$,
and hence
it follows from $f(0) = F(0) = 0$ and
the mean value property of harmonic functions
that both hand sides of (\ref{ineq:Littlewood})
reduces to $2 \pi u(0)$.

Suppose that equality holds in (\ref{ineq:Littlewood}).
Then for almost every $\theta$,
$u \circ F (\omega(re^{i \theta }))= U (\omega(re^{i \theta }))$
holds.
Without loss of generality we may assume that $F$ is not constant.
By the classical Schwarz lemma it suffices  to show that
$u$ is harmonic in $F({\mathbb D}(0,r))$
when $| \omega (z) | < |z|$ for all $z \in {\mathbb D}$,
since otherwise $\omega (z) = \varepsilon z$ in ${\mathbb D}$
for some $\varepsilon \in \partial {\mathbb D}$.
Therefore for any fixed real $\theta$,
$\omega (re^{i \theta }) $ is an interior point of
${\mathbb D}(0,r)$.
It follows from the maximum principle
for subharmonic functions that $u \circ F = U$ in
${\mathbb D}(0,r)$.

Now we show that $u$ is harmonic in $F({\mathbb D}(0,r))$.
Let $w_0 \in F({\mathbb D}(0,r))$ and choose $z_0 \in {\mathbb D}(0,r)$
with $F(z_0) = w_0$. If $F'(z_0) \not= 0 $,
then $F$ maps a neighborhood $V_{z_0}$ of $z_0$ conformally onto
a neighborhood $F(V_{z_0})$ of $w_0$,
and hence $u = U \circ (F|_{V_{z_0}})^{-1}$
is harmonic  in $F(V_{z_0})$.
Even if $F'(z_0) = 0$,
$u$ is at least continuous at $z_0$.
In fact, for each $\eta > 0$ there exists
$\delta > 0$ such that
$|U(z)-U(z_0)| < \delta$ holds for  $z \in {\mathbb D}(z_0, \delta )$.
Since $F$ is nonconstant, $F({\mathbb D}(z_0, \delta ))$ is an open
neighborhood of $w_0$ and $|u(w)-u(w_0)| < \eta $ holds
for $w \in F({\mathbb D}(z_0, \delta ))$.
Thus $u$ is continuous at $w_0$.

We have shown that $u$ is continuous in $F({\mathbb D}(0,r))$ and
harmonic in $F ({\mathbb D}(0,r))$ except
at each point in the set of critical values
\[
   B
   =
   \{ w_0 = F(z_0) :\, z_0 \in {\mathbb D}(0,r) \;
   \text{with} \; F'(z_0) = 0 \} .
\]
Since $B$ is finite,
each point in $B$ is isolated and hence
is a removable singularity of $u$.
Thus  $u$ is harmonic in $F (V_{z_0})$.
\endproof

\begin{proof}[Proof of Theorem \ref{MainTheorem}]
Let $f \in {\mathcal C} ( \Omega ) $
and $h(z) = 1+zf''(z)/f'(z)$.
Then $h$ satisfies
$h(0)=1$,
$h \prec \phi_\Omega$.
By the starlikeness of $\Omega$ with respect to $1$
it follows from the Suffridge lemma
(see \cite{Suffridge}) that
\[
      \log f'(z)
   =
     \int_0^z \frac{h(\zeta)-1}{\zeta} \, d \zeta
   \prec
      \int_0^z \frac{\phi_\Omega (\zeta)-1}{\zeta} \, d \zeta
   = \log k_\Omega' ( z) ,
\]
where $k_\Omega$ is defined by \eqref{eq-ext3}.
Also it follows from the starlikeness of $\Omega$ with respect to $1$
that $\log k_\Omega'$ is convex univalent in ${\mathbb D}$.
Now the latter half of the statement is a direct consequence
of Lemma \ref{lemma:subordination_and_mean}.
\end{proof}

\begin{proof}[Proof of Theorem \ref{thm:arclength}]
Let $u(w) = r e^{\text{\rm Re}\, w }$
and $L$ be the Laplace operator.
Since $L( r e^{\text{\rm Re}\, w }) = r e^{\text{\rm Re}\, w } > 0$,
$u$ is subharmonic in $\mathbb{C}$ and
(\ref{ineq:arclenght})
easily follows from Theorem \ref{MainTheorem}.
Furthermore
the subharmonic function $r e^{\text{\rm Re}\, w }$ is not harmonic
in any domain in ${\mathbb C}$.
Thus if equality holds in (\ref{ineq:arclenght}),
then $\log f'(z) = \log k_\Omega' ( \varepsilon z)$ for
some $\varepsilon \in \partial {\mathbb D}$
and hence
we obtain
that $f(z) = \overline{\varepsilon} Q_\Omega ( \varepsilon z)$.
\end{proof}

\begin{proof}[Proof of Corollary \ref{cor:several-ineqalities}]
Inequalities  (\ref{ineq:loglog}), (\ref{ineq:p-th-log}),
(\ref{ineq:Philog}) and (\ref{ineq:Phiarg})
are  consequences of Theorem \ref{MainTheorem}
and the subharmonicity of functions
$\log |w| $, $|w|^p$, $\Phi ( \pm \text{\rm Re} \, w )$ and
$\Phi ( \pm \text{\rm Im} \, w )$,
respectively.

Notice $\log |w|$  is not harmonic
in $\log k_\Omega'({\mathbb D}(0,r))$
for any $r \in (0,1)$ because
$\log k_\Omega'(0) = 0$.
Also $|w|^p$ is not harmonic in any domain in
 ${\mathbb C}$.
Furthermore
$\Phi ( \pm \Real w )$
and $\Phi ( \pm \Imaginary w )$,
are  not harmonic in
$(\log k_\Omega ') (\mathbb{D})$, since
$\Phi (\pm t )$ is not linear
in the interval given by (\ref{log-interval}) or by (\ref{arg-interval}).
Thus if equality holds in (\ref{ineq:loglog}),
(\ref{ineq:p-th-log}), (\ref{ineq:Philog}) or (\ref{ineq:Phiarg}),
then by
Lemma \ref{lemma:subordination_and_mean}
we have
$ \log f'(z) = \log k_\Omega' (\varepsilon z )$
for some $\varepsilon \in \partial {\mathbb D}$
and hence $f(z) = \overline{\varepsilon} k_\Omega(\varepsilon z )$.
\end{proof}

\section{Circular Symmetrization}
We summarize without proofs some of the standard facts
on the theory of $*$-functions developed by
Baernstein \cite{Baernstein}.
For more on $*$-functions we refer to Duren \cite{DurenBook}.

Let $|E|$ denote the Lebesgue measure of a Lebesgue measurable
set $E \,( \subset {\mathbb R})$.
Let $h :\, [-\pi, \pi ] \rightarrow {\mathbb R} \cup \{ \pm \infty \}$
be a Lebesgue measurable function which is finite-valued almost
everywhere.
Then the distribution function $\lambda_h$ defined by
\[
   \lambda_h (t)
   =
   | \{ h > t \}|
   =
   | \{ \theta \in [-\pi, \pi ] :\, h( \theta )  > t \}|
\]
is nonincreasing and right continuous on ${\mathbb R}$,
and satisfies $\lim_{t \rightarrow - \infty} \lambda_h(t) = 2 \pi$
and $\lim_{t \rightarrow \infty } \lambda_h(t) = 0$. Let
\[
  \hat{h}( \theta )
  =
  \left\{
  \begin{array}{ll}
    \inf \{ t \in {\mathbb R} :\, \lambda_h(t) \leq 2 \theta \}
    & \mbox{ for }~ | \theta | < \pi
    \\[1ex]
    \essinf h
    &  \mbox{ for }~| \theta | = \pi .
  \end{array}
 \right.
\]
It is easy to see that  $\hat{h}$ is symmetric, i.e.,
$\hat{h} ( - \theta ) = \hat{h}( \theta )$,
and  satisfies the following conditions:
\begin{itemize}
\item[{\rm (i)}]
$\essinf h \leq \hat{h}( \theta ) \leq \esssup h$
\item[{\rm (ii)}]
$\hat{h}$ is right continuous and nonincreasing on $[0, \pi]$
\item[{\rm (iii)}]
$\lim_{\theta \downarrow 0} \hat{h} ( \theta )
= \hat{h}(0) = \esssup h$ and
$\lim_{\theta \uparrow 2 \pi} \hat{h} ( \theta )
= \hat{h}(\pi ) = \essinf h$
\item[{\rm (iv)}]
$\hat{h}$ is equimeasurable with $h$, i.e.,
$\lambda_{\hat{h}}(t) = \lambda_h (t)$ for all $t \in {\mathbb R}$.
\end{itemize}
The function $\hat{h}$ is called
the symmetric nonincreasing rearrangement of $h$.
Notice that $\hat{h}$ is unique in the sense that
if $\tilde{h}$ is also a symmetric function on
$[-\pi, \pi]$, nonincreasing and right continuous
on $[0, \pi ]$ with $\tilde{h}( \pi ) = \essinf h$
and equimeasurable with $h$, then $\tilde{h} = \hat{h}$.

For $h \in L^1[-\pi, \pi]$,
the $*$-function of $h$ is the function
defined by
\begin{equation}
h^* ( \theta ) = \sup_{|E| = 2 \theta } \int_E h( s ) \, ds ,
\quad 0 \leq \theta \leq \pi ,
\label{eq:def-of-*function}
\end{equation}
where supremum is taken over all Lebesgue measurable subsets
of $[-\pi, \pi]$ with $|E| = 2 \theta $. Then it is known that
\begin{equation}
 h^* ( \theta ) = \int_{-\theta}^\theta \hat{h}(s) \, ds ,
 \quad 0 \leq \theta \leq \pi .
\label{eq:*function-and-symmetrization}
\end{equation}

\begin{lemma}{\rm (\cite[p.~150]{Baernstein})}
For $h$, $H \in L^1[-\pi, \pi ]$,
the following statements are equivalent:
\begin{itemize}
 \item[{\rm (a)}]
For every convex nondecreasing function $\Phi$ on ${\mathbb R}$
\[  \int_{-\pi}^\pi \Phi(h(\theta )) \, d \theta
  \leq \int_{-\pi}^\pi \Phi(H(\theta )) \, d \theta ,
\]
 \item[{\rm (b)}] For every $t \in {\mathbb R}$
\[  \int_{-\pi}^\pi (h(\theta )-t)^+ \, d \theta
  \leq \int_{-\pi}^\pi (H(\theta )-t)^+ \, d \theta ,
\]
\item[{\rm (c)}] $h^* ( \theta ) \leq H^* ( \theta )$ \quad\mbox{ for } $0 \leq \theta \leq \pi$,
\end{itemize}
\label{lemma:Baernstein}
where  $(h(\theta)-t)^+ = \max \{ h(\theta)-t , 0 \}$.
\end{lemma}

Let $v$ be a subharmonic function
in the unit disk ${\mathbb D}$.
Then for each fixed $r \in (0,1)$,
$v(re^{i \theta})$ is an integrable function of
$\theta \in [- \pi, \pi ]$.
Let $\hat{v}(re^{i \theta })$ and $v^* (re^{i \theta })$
be the symmetrically nonincreasing rearrangement
and the $*$-function
of the function $[-\pi, \pi] \ni \theta \mapsto v(re^{i \theta })$,
respectively.
The function $\hat{v}$ is called the circular symmetrization of $v$.

We now can conclude an inequality concerning $*$-functions from
Lemma \ref{lemma:subordination_and_mean}. The following lemma is not knew and it is an equivalent variant of
Lemma 2 in Leung \cite{Leung}.
\begin{lemma}
Let $f , F \in \mathcal{A}$ with
$f \prec F$. Then for any subharmonic function
$u$ in $ F ( {\mathbb D} )$ and $r \in (0,1)$
\[  (u \circ f)^*(re^{i \theta})
\leq (u \circ F)^*(re^{i \theta}), \quad 0 \leq \theta \leq \pi .
\]
\label{lemma:Leung}
\end{lemma}
\begin{proof}
Since $(u(w)-t)^+$ is also subharmonic
in $F({\mathbb D})$ for any $t \in {\mathbb R}$,
we have by Lemma \ref{lemma:subordination_and_mean} that
\[  \int_{-\pi}^\pi (u\circ f (re^{i \theta })-t)^+ \, d \theta
 \leq \int_{-\pi}^\pi (u\circ F (re^{i \theta })-t)^+ \, d \theta .
\]
Therefore it follows from Lemma \ref{lemma:Baernstein} that
$(u \circ f)^*(re^{i \theta})  \leq (u \circ F)^*(re^{i \theta})$
for $0 \leq \theta \leq \pi$.
\end{proof}

\begin{proof}[Proof of Theorem \ref{thm:localization}]
Let $\Phi $ be a nondecreasing convex function in ${\mathbb R}$.
Then $\Phi ( \text{\rm Re} \, w )$ is a subharmonic function
of $w \in {\mathbb C}$. By Lemma \ref{lemma:Leung} we have
\begin{equation*}
  (\Phi ( \log |f '|))^*  (re^{i \theta }) \leq
( \Phi ( \log |k_\Omega'|))^*  (re^{i \theta })
\label{ineq:*functions}
\end{equation*}
for all $r \in (0,1)$ and  $\theta \in [0, \pi]$.
Now we temporarily suppose that
for fixed $r \in (0,1)$, $\log |k_\Omega'(re^{i \theta })|$
is a symmetric function of $\theta$ and nonincreasing on
$[0, \pi]$.
Then so is $\Phi ( \log | k_\Omega' (re^{i \theta }) | )$,
since $\Phi$ is nondecreasing.
Thus the symmetrically nonincreasing rearrangement of
$\Phi (\log | k_\Omega'| )$ coincides with itself,
i.e.,
$(\Phi (\log | k_\Omega'|))^{\widehat{}} (re^{i \theta })
= \Phi (\log | k_\Omega' (re^{i \theta })|)$.
Therefore we have by (\ref{eq:def-of-*function}) and
(\ref{eq:*function-and-symmetrization})
that for any Lebesgue measurable set $E \subset [-\pi ,\pi ]$
with $|E| = 2 \theta $
\begin{align*}
 \int_E
 \Phi ( \log |f'(re^{i s })|) \, ds
 &\leq
 \Phi ( \log |f '|)^*  (re^{i \theta })
\\
 &\leq
 \Phi ( \log |k_\Omega'|)^*  (re^{i \theta })
\\
 &=
 \int_{- \theta }^\theta
  \Phi (\log | k_\Omega'|)^{\widehat{}} (re^{i s }) \, ds
 =
 \int_{- \theta }^\theta
 \Phi (\log | k_\Omega' (re^{i s })|) \, ds  .
\end{align*}

It remains to show
that the function $ \theta \mapsto \log | k_\Omega'(re^{i \theta})|$ is
symmetric and strictly decreasing on $[0, \pi ]$.
Since $\Omega$ is symmetric with respect to ${\mathbb R}$,
we have $\overline{\phi_\Omega ( \overline{z})} = \phi_\Omega (z)$.
This implies
$\log | k_\Omega'(re^{ - i \theta})| = \log | k_\Omega'(re^{ i \theta})|$.
Furthermore from  $\phi_\Omega '(0) > 0$ it follows that
$\phi_\Omega$ maps the upper half disk
${\mathbb D} \cap \{ z\in\mathbb{C} :\, \Imaginary z > 0 \}$
conformally onto
$\Omega \cap \{ w\in\mathbb{C}:\, \Imaginary w > 0 \}$.
Thus for $\theta \in (0, \pi )$
\begin{align*}
  \frac{d}{d \theta }
   \left\{ \log | k_\Omega'(re^{i \theta }) | \right\}
&=
  \frac{d}{d \theta }
  \Real \left\{ \log  k_\Omega'(re^{i \theta }) \right\}
\\
&=  \Real \left\{  \frac{d}{d \theta } \log k_\Omega'(re^{i \theta })
 \right\}
\\
&=  \Real \left\{  \frac{i re^{i \theta } k_\Omega''(re^{i \theta })}
  {k_\Omega'(re^{i \theta }) }  \right\}
\\
&= \Real \left\{   i (\phi_\Omega ( re^{i \theta })-1) \right\}
 = - \Imaginary \phi_\Omega ( re^{i \theta }) < 0 .
\end{align*}
\end{proof}

\bibliographystyle{amsplain}

\end{document}